\documentclass{amsart}
\usepackage{amsmath,amssymb,epic,graphicx, mathrsfs,enumerate,xcolor}
\usepackage{amsfonts}
\usepackage{amsthm}
\usepackage{amssymb}
\usepackage{latexsym}
\usepackage{longtable}
\usepackage{epsfig}
\usepackage{amsmath}
\usepackage{hhline}
\usepackage{float}
\usepackage{hyperref}
\usepackage{comment}
\restylefloat{table}

\input xy
\xyoption{all}

\newcommand{\red}[1]{\textcolor{red}{#1}}
\newcommand{\blue}[1]{\textcolor{blue}{#1}}

\newtheorem{theorem}{Theorem}[section]

\newtheorem{cor}[theorem]{Corollary}
\newtheorem{prop}[theorem]{Proposition}

\newtheorem{lemma}[theorem]{Lemma}

\newtheorem{rem}[theorem]{Remark}
\newtheorem{problem}[theorem]{Problem}
\newtheorem{exa}[theorem]{Example}

\newtheorem{innercustomgeneric}{\customgenericname}
\providecommand{\customgenericname}{}
\newcommand{\newcustomtheorem}[2]{%
  \newenvironment{#1}[1]
  {%
   \renewcommand\customgenericname{#2}%
   \renewcommand\theinnercustomgeneric{##1}%
   \innercustomgeneric
  }
  {\endinnercustomgeneric}
}

\newcustomtheorem{customthm}{Theorem}
\newcustomtheorem{customprob}{Problem}

\newcommand{\semid}{\rtimes}                          
\newcommand{\nor}{\trianglelefteq}         
\newcommand{\udot}{^{{}^{\textbf{ .}}}}               
\newcommand{\abs}[1]{\left\vert#1\right\vert}         
\newcommand{\set}[1]{\left\{#1\right\}}               
\newcommand{\seq}[1]{\left<#1\right>}                 
\newcommand{\lto}{\longrightarrow}

\newcommand{\lmto}{\longmapsto}

\newcommand{\GL}{{\operatorname{GL}}}
\newcommand{\SL}{{\operatorname{SL}}}
\newcommand{\PGL}{{\operatorname{PGL}}}
\newcommand{\GammaL}{{\operatorname{\Gamma L}}}
\newcommand{\PGammaL}{{\operatorname{P\Gamma L}}}
\newcommand{\PSL}{{\operatorname{PSL}}}
\newcommand{\AGL}{{\operatorname{AGL}}}
\newcommand{\ASL}{{\operatorname{ASL}}}
\newcommand{\ASigmaL}{{\operatorname{A\Sigma L}}}
\newcommand{\AGammaL}{{\operatorname{A\Gamma L}}}

\newcommand{\PGU}{{\operatorname{PGU}}}
\newcommand{\PGammaU}{{\operatorname{P\Gamma U}}}
\newcommand{\PSU}{{\operatorname{U}}}
\newcommand{\Sp}{{\operatorname{Sp}}}
\newcommand{\PSp}{{\operatorname{PSp}}}

\DeclareMathOperator{\Aut}{Aut}
\DeclareMathOperator{\Out}{Out}
\DeclareMathOperator{\Inn}{Inn}
\DeclareMathOperator{\soc}{soc}

\begin{document}

\title{On some questions related to integrable groups}

\author{Russell D. Blyth} 
\address{Department of Mathematics and Statistics, Saint Louis University, 220 N. Grand Blvd., St. Louis, MO 63103, USA}
\email{\href{mailto:russell.blyth@slu.edu}{russell.blyth@slu.edu}}
\author{Francesco Fumagalli} 
\address{Dipartimento di Matematica e Informatica 'Ulisse Dini', 
Viale Morgagni 67/A, 50134 Firenze, Italy}
\email{\href{mailto:francesco.fumagalli@unifi.it}{francesco.fumagalli@unifi.it}}
\author{Francesco Matucci} 
\address{Dipartimento di Matematica e Applicazioni, Universit\`{a} degli Studi di Milano--Bicocca, Milan 20125, Italy.}
\email{\href{mailto:francesco.matucci@unimib.it}{francesco.matucci@unimib.it}}

\thanks{The first author gratefully acknowledges the support of the  Universit\`a 
degli Studi di Firenze  and the Stolle Fund of the College of Arts and Sciences, Saint Louis University. The second and third authors are members of the Gruppo Nazionale per le Strutture Algebriche, Geometriche e le loro Applicazioni (GNSAGA) of the Istituto Nazionale di Alta Matematica (INdAM) and the third author gratefully acknowledges the support of the 
Funda\c{c}\~ao para a Ci\^encia e a Tecnologia (CEMAT-Ci\^encias FCT projects UIDB/04621/2020 and UIDP/04621/2020) and of the Universit\`a degli Studi di Milano--Bicocca
(FA project ATE-2017-0035 ``Strutture Algebriche'').}
\date{}



\begin{abstract}A group $G$ is \emph{integrable} if it is isomorphic to the derived 
subgroup of a group $H$; that is, if $H'\simeq G$, and in this case $H$ 
is an \emph{integral} of $G$. 
If $G$ is a subgroup of $U$, we say that $G$ is 
\emph{integrable within $U$} if $G=H'$ for some $H\leq U$.
In this work we focus on two problems 
posed in \cite{ACCM1}. We classify the almost-simple finite groups $G$ that are integrable, 
which we show to be equivalent to those integrable within $\Aut(S)$, where $S$ is the socle of $G$. 
We then classify all $2$-homogeneous 
subgroups of the finite symmetric group $S_n$ 
that are integrable within $S_n$.
\end{abstract}
\maketitle
\begin{center}
\textit{This paper is dedicated to the memory of Carlo Casolo}
\end{center}

\section{introduction}

Recently several articles (\cite{FM17,ACCM1,ACCM2}) have appeared on the topic of the integrability of a group $G$, where we say that a group $H$ is an \emph{integral} 
of a group $G$ if $G \simeq H'$ and we then
say that $G$ is \emph{integrable}. Given two groups $G \le U$,
we say that $G$ is \emph{(relatively) integrable within} $U$ if there exists a subgroup $H \le U$ such that $H'=G$. 
Burnside \cite{B1913} was the first to consider integrals of groups, showing, for example, that a nonabelian finite $p$-group with cyclic center cannot have an integral which is a finite $p$-group. 
Since then, several other authors 
have considered which groups can appear as derived subgroups under certain restrictions.
Among known results, we mention that all abelian groups are integrable, while all direct powers of dihedral groups
$D_{2n}$ are non-integrable for $n \ge 3$, and that if a finite group has an integral, it also has a finite integral. We leave the reader to explore these results and other prior work (see
\cite{ACCM1, ACCM2} and their references). 

In \cite[Section 8.2]{ACCM1} the following problem is posed:

\begin{customprob}{1}
\label{p:almost-simple}
Classify all the almost-simple finite groups 
$G$ that are integrable within $\Aut(S)$, 
where $S$ denotes the socle of $G$.
\end{customprob}

More generally, we can ask the following question.

\begin{customprob}{1$'$}
\label{p:almost-simple-b}
Classify all the integrable almost-simple finite groups.
\end{customprob}

In Section \ref{s-almost-simple} we provide the following response.

\begin{customthm}{A}
\label{thm:almost_simple}
A finite almost-simple group $G$ (with socle $S$) is integrable if and only 
if it is integrable within $\Aut(S)$.
Moreover, 
the subgroups between $S$ 
and $\Aut(S)$ 
that are integrable within $\Aut(S)$ are precisely those contained in $\Aut(S)'$.
\end{customthm}

In Section \ref{S-2-homog} we consider the relative integrability of $2$-homogeneous groups. A subgroup $G$ acting on a set $\Omega$ 
is said to be \emph{$k$-homogeneous} (or 
\emph{$k$-set transitive}) if it is 
transitive on the set $\Omega^{\set{k}}$  
of all $k$ element subsets of 
$\Omega$ (for $k\geq 1$). 
We are interested in the following problem, which is also mentioned in \cite[Section 8.2]{ACCM1}.

\begin{customprob}{2}\label{p:2-homo}
Classify all the $2$-homogeneous 
subgroups $G$ 
of the finite symmetric group $S_n$ that are 
integrable within $S_n$.
\end{customprob}

In Section \ref{S-2-homog} we prove

\begin{customthm}{B}\label{thm:main_int_2-homo}
Let $G$ be a $2$-homogeneous subgroup of the 
finite symmetric group $S_n$ and let $S$ be its socle. 
If $G$ is integral within $S_n$ then $G$ lies between $(N_{S_n}(S))''$ and $(N_{S_n}(S))'$; the converse is also true with the exception of a few solvable groups (see Remark  
\ref{rem:thmB+exceptions}). 
\end{customthm}
 
The proof of Theorem \ref{thm:main_int_2-homo} 
is based on the classification of the $2$-homogeneous
subgroups of $S_n$ (see Theorem 
\ref{t:2-homo_classification}) and it is completed by 
considering each possible case appearing in that classification. The proof, of course, 
relies on the classification of finite simple groups.

We note that the group $\PGL_3(7)$, in its 
action on the $57$ projective points, is an 
example of a $2$-homogeneous subgroup of $S_{57}$ that is 
not integrable within $S_{57}$. Nevertheless, 
the group $\PGL_3(7)$ is 
integrable as an abstract group (see Example 
\ref{ex:PGammaL_3(7)}). 


Most of our notation is standard and well-known.
We write actions on the right, and for $x, y$ in a group $G$, we define the conjugation $x^y$ to be $y^{-1}xy$ and the commutator $[x,y]$ to be $x^{-1}y^{-1}xy$.
We use the ATLAS \cite{Atlas} notation for some groups and constructions. Hence we use $A.B$ to denote any group that has a normal subgroup isomorphic to $A$ for which the corresponding quotient group is isomorphic to $B$ (an (upward) extension of $A$ by $B$), and we use $A\!:\!B$ to indicate a case of $A.B$ that is a split extension  $A \semid B$ (a notation we also sometimes use), while $A\!\udot\! B$ denotes a non-split extension.  
Also, given a  positive integer $n$ we denote the cyclic group of order $n$ with both symbols $C_n$ and $n$. 

Finite affine 
semilinear transformation groups $\AGammaL_d(q)$, 
and their important 
subgroups (such as 
$\ASL_d(q)$ and $\AGL_d(q)$) arise 
in Section \ref{S-2-homog}; as a reference, 
we suggest \cite[Section 2.8]{DM}. 

\section{Preliminaries}

We start with a simple example that draws a clear distinction between the integrability of a group and the integrability of a particular subgroup within a particular group.

\begin{exa}
{\rm Consider the dihedral group $D_8 = \langle r,s \mid r^4=s^2=1, srs^{-1}=r^{-1}\rangle$.
Since $D_8'=\{1,r^2\}$, we see immediately that $D_8'$ is the unique subgroup of order $2$ of $D_8$ to be integrable within $D_8$, while the other four subgroups of order $2$ are not integrable within $D_8$ (although they are all integrable since they are abelian).}
\end{exa}


The two parts of the next result are fundamental in reducing integrability within certain groups to integrability within suitable quotients.

\begin{lemma}\label{l:basic}
Let $T$ be a finite group and $G$ be a subgroup of $T$.
\begin{enumerate}
\item If $G$ is not a cyclic $p$-group and it
contains a 
unique minimal normal subgroup $S$, then $G$ is 
integrable within $T$ if and only if $G/S$ is 
integrable within $N_{T}(S)/S$.
\item If $G$ admits a nontrivial characteristic 
perfect subgroup $K$, then $G$ is integrable within $T$ if 
and only if $G/K$ is integrable within $N_{T}(K)/K$.
\end{enumerate}
\end{lemma}
\begin{proof} 
(1) Let $H\leq T$ be such that $H'=G$. Of course $S$, 
since it is characteristic in $H'$, is normal in $H$, 
that is, $H\leq N_{T}(S)$. Then we have that 
$H/S$ is a subgroup of $N_{T}(S)/S$ having derived subgroup equal to $G/S$.
Conversely, assume that $H/S\leq N_{T}(S)/S$ 
is such that $(H/S)'=G/S$. Thus $G=H'S \le HS$, and since $S \le H$, we have
$G \le H$. 
Since $S$ is the unique minimal normal 
subgroup of $G$ and $H'\nor G$, we have that 
either $H'=1$ or $S\leq H'$. If $H'=1$ then, since $G\leq H$,  $G$ is abelian 
and therefore $G$ is a cyclic $p$-group for some prime $p$, which is a
contradiction. Thus $S \leq H'$, and so $G=H'$ is 
integrable within $T$.

(2) If  $G=H'$ for some $H\leq T$, then since $K$ 
is characteristic in $H'$, it is normal in $H$, 
that is, $H\leq N_T(K)$. But then $H/K$ is a 
subgroup of $N_T(K)/K$ such that $(H/K)'=G/K$.
Conversely, assume now that $G/K=(H/K)'$ for some 
$H\leq N_T(K)$. Then $G=H'K$. Since $K\leq H$ and 
$K$ is perfect, we have $K=K'\leq H'$ and so $G=H'$.
\end{proof} 


We will need the following result, which does not require the involved groups to be finite.

\begin{lemma}\label{l:metacyclic}
Let $A$ be a 
metacyclic group 
$A=\seq{x,y}$ with $\seq{x}\nor A$.  
The subgroups of $A$ that are 
integrable within $A$ are precisely 
all the subgroups of $A'$.   
\end{lemma}
\begin{proof}
Let $B$ be any subgroup of $A$ that is integrable 
within $A$. Then there exists some $C\leq A$ such 
that $B=C'$, and therefore $B\leq A'$. 
For the converse, write $[x,y]=x^c$ 
for some $c \in \mathbb{Z}$, so that $A'=\seq{x^c}$ by standard commutator calculus. 
Since $[x,y]$ commutes with $x$, we have
$[x^t,y]=x^{ct}$
for some $t \in \mathbb{Z}$. Therefore if $B$ is an arbitrary subgroup of $A'$, 
say $B=\seq{x^{ct}}$, then if we set $C=\seq{x^t,y}$ 
we have $B=C'$.
\end{proof}

The following example shows that 
Lemma \ref{l:metacyclic} cannot be extended to the classes of 
supersolvable or nilpotent groups. 
\begin{exa}\label{ex:D_8wrC_2}
{\rm Let  $G$ be the standard wreath product $G=D_8\wr C_2$. Then  $G'\simeq D_8\times C_2$, and therefore not all 
subgroups of $G'$ are integrable, since $D_8$ is not integrable (see \cite{ACCM1} for a history of this result). In this particular case
the subgroups of $G'$ that are not integrable in $G$ are:
four copies of $D_8$ (all maximal in $G'$) and three copies 
of $C_2\times C_2$, these last three subgroups are each contained 
in a unique maximal subgroup of $G'$ of order $8$
(one $C_4\times C_2$ and two $(C_2)^3$).}
\end{exa}

\section{integrable almost-simple groups}\label{s-almost-simple}

We recall that \cite[Theorem 3.2]{ACCM2} states that if a group $G$ is integrable, then 
$\Inn(G) \le \Aut(G)'$ and that, indeed, $\Inn(G)$ has an integral within $\Aut(G)$.
Following the language of ~\cite[Section 5]{ACCM1}, we recall that a group $H$ is a \emph{reduced integral} of a group $G$ if $H'=G$ and $C_H(G)=1$. The following result uses ideas from  \cite[Lemma~5.2]{ACCM1} and  \cite[Theorem~3.2]{ACCM2} and does not require the groups to be finite.


\begin{theorem}
\label{p:almost_simple_tot_int}
Let $S$ be a 
nonabelian simple group and let $G$ be a group such that 
$\Inn(S)\leq G\leq \Aut(S)$. Then $G$ is integrable if 
and only if $G$ is integrable within $\Aut(S)$.
\end{theorem}
\begin{proof}
The converse implication is obvious, so we assume now 
that $G$ is integrable.

We recall that we write functions on the right, 
that is, if $g$ is an 
arbitrary element of $G$ and $s$ is an 
arbitrary element of $S$ we write $s.g$ for the image of 
$s$ under the automorphism $g$. Also, if $A$ is any group 
and $a\in A$, denote with $\gamma_a\in \Aut(A)$ the 
conjugation map $x.\gamma_a=a^{-1}xa$ for all $x \in A$. 
Since for every $g\in G$ and 
every $s\in S$ we have that 
$g^{-1}\gamma_sg=\gamma_{s.g}$,
the action of $G$ on $S$ is equivalent to the action by 
conjugation of $G$ on its normal subgroup $\Inn(S)$.
In particular, $C_G(\Inn(S))=1$, and $\Inn(S)$ 
is characteristic in $G$ 
(since $\Inn(S)$ is the unique minimal normal 
subgroup of $G$). 
Since $Z(G)\leq C_G(\Inn(S))=1$, we  apply 
\cite[Lemma 5.2]{ACCM1} to deduce that $G$ admits a 
reduced integral $H$, namely  a group $H$ 
such that $H'=G$ and $C_H(G)=1$. We extend the 
faithful action of $G$ on $S$ to a faithful action 
of $H$ on $S$ by defining, for every $h\in H$ and every 
$s\in S$, the element $s.h$ to be the unique element of $S$ 
such that $h^{-1}\gamma_sh=\gamma_{s.h}$. Note that this extension of the action
is well defined since $\Inn(S)$ is characteristic in $G$ 
and $H$ acts faithfully by conjugation on $G$. Let 
$K=C_H(S)$ denote the kernel of this action. Then $K$ is 
normal in $H$ and $K\cap G=C_G(S)=1$. Therefore 
$[K,G]=1$, but then $K\leq C_H(G)=1$. This shows that 
the above action of $H$ on $S$ is faithful, that is $H$ 
is an integral of $G$ within $\Aut(S)$.
\end{proof}

Throughout this section let $S$ be a finite nonabelian
simple group and, by identifying $S$ with 
$\Inn(S)$, let $S\leq G\leq \Aut(S)$. 
By applying Lemma \ref{l:basic} (case (1) with 
$T:=\Aut(S)=N_T(S)$) we 
have the following crucial fact. 
\begin{lemma}\label{l:almost-simple_1}
$G$ is integrable within $\Aut(S)$ if and only if 
$G/S$ is integrable in the group of outer automorphisms 
$\Out(S)=\frac{\Aut(S)}{\mathrm{Inn}(S)}$.
\end{lemma}
As an immediate consequence of Lemma \ref{l:almost-simple_1}
and Lemma \ref{l:metacyclic} we have that 
Theorem \ref{thm:almost_simple} holds for all those $S$ such that $\Out(S)$ is either abelian or metacyclic. 

To complete our classification of integrable 
almost-simple groups we need only focus on 
those $S$ for which $\Out(S)$ is neither abelian
nor metacyclic. 
Using the classification of finite simple groups (CFSG), such simple groups (and their 
automorphism groups) are summarized in the following 
proposition. 
\begin{prop}\label{p:nonab_out} 
Let $S$ be a finite simple group having outer 
automorphism group that is neither abelian nor metacyclic. 
Then $(S,\Out(S))$ must be contained in the following list:
\begin{enumerate}
\item $\big(A_n(q), d.(f\times 2)\big)$, 
where $n\geq 2$, $d=(n+1,q-1)>1$, $q=p^f$ and $f>1$ 
is even.
%
%
%
%
\item $\big(D_4(q), d.\big(f\times S_3)\big)$, where $d=(2,q-1)^2$ 
and $q=p^f$.
\item  $\big(D_n(q), d.\big(f\times 2)\big)$, where $n>4$, $q=p^f$ with $p$ an odd prime, and, respectively,
$d=(2,q-1)^2= 2^2$ when $n$ is even, and 
$d=(4,q^n-1)$ when $n$ odd. 
%
%
%
%
%
%
%
%
\item 
$\big(E_6(q), 3.(f\times 2)\big)$, where $d=(3,q-1)=3$, 
$q=p^f$ and $f$ is even.
%
%
\end{enumerate}
\end{prop}
\begin{proof}
This result is a consequence of the CFSG. The ATLAS \cite[Tables 1 and 5]{Atlas} is a good reference. 
Note that a normal subgroup 
of order two in any group is always central,
which implies that the 
outer automorphism group of each of the following groups 
is always abelian or metacyclic: 
$A_1(q)$, $B_2(q)$ for $q$ odd, $B_n(q)$ and $C_n(q)$ 
when $n\geq 3$ and $E_7(q)$. Also note that field 
automorphisms commute with graph automorphisms 
(see \cite{Carter}).
\end{proof}


We complete the proof of Theorem \ref{thm:almost_simple}
by considering these four remaining cases of Proposition \ref{p:nonab_out}. 
For convenience in the following we set $A=\Out(S)$.\\ 

Case (1). $S=A_n(q)=\PSL_{n+1}(q)$, 
with $n\geq 2$, $d=(n+1,q-1)>1$ and $q=p^f$ 
with $f>1$ even.\\ 
In this situation the group $A$ is metabelian isomorphic to $d.(f\times 2)$ and it has the following 
presentation (see \cite[Proposition 2.2.3]{KL}):
$$\seq{\delta,\phi,\iota \,\vert\, 
\delta^d=\phi^f=\iota^2=1, \delta^\phi=\delta^p, 
\delta^\iota=\delta^{-1}, 
\phi^\iota=\phi}.$$
We have that $A'\leq \seq{\delta}$,  
since $A/\seq{\delta}$ is abelian. Moreover $A'\geq 
\seq{\delta^2, \delta^{p-1}}$, as $[\delta,\iota]=
\delta^{-2}$ and $[\delta,\phi]=\delta^{p-1}$.\\ 
Assume first that $p$ is odd. Then 
$\delta^{p-1}$ is a power of $\delta^2$, forcing
$\seq{\delta^2, \delta^{p-1}}=\seq{\delta^2}$. Since $A/\seq{\delta^2}$
is abelian when $p$ is odd, then $A'=\seq{\delta^2}$.
Moreover if we set 
$B=\seq{\delta,\iota}$ then $B$ is a metacyclic group 
whose derived subgroup is $B'=\seq{\delta^2}=A'$. 
Thus by Lemma \ref{l:metacyclic} every subgroup of $A'$ 
is integrable in $B$ and thus also in $A$.\\ 
Now we assume $p=2$. Since $\seq{\delta^2, \delta^{p-1}} \leq A'$, then $A'=\seq{\delta}$.
Arguing as in the preceding 
situation, every subgroup of $\seq{\delta}$ is 
integrable in $\seq{\delta,\iota}$ and thus in $A$ too.

We conclude that every subgroup of $A'$ is 
integrable in $A$.\\

Case (2). $S=D_4(q)=P\Omega_8^+(q)$.\\ 
By \cite[p. 181]{Kle} we have that 
$$\Out(S)\simeq \begin{cases} S_3\times f & \textrm{ if $q$ is even,}\\
S_4\times f & \textrm{ if $q$ is odd.} 
\end{cases}$$
So when $q$ is even all subgroups of $A'$ 
are integrable (these subgroups are just $1$ and $A'=A_3$, 
whose order is three).
When $q$ is odd then $A'=A_4$ and again all subgroups $H$ of 
$A'$ are integrable in $A$, since if 
$H=A'$ this is immediate, if $H=V_4$ then $H=A''$, if $H$ has 
order $3$ then $H=(S_3)'$ for some $S_3< S_4<A$, and, finally, 
if $H$ has order two then $H=(D_8)'$ for some 
$D_8<S_4<A$. \\

Case (3). 
$S=D_n(q)=O^+_{2n}(q)=P\Omega^+_{2n}(q)$, with $n>4$.\\
When the group $A$ is nonabelian and not metacyclic then,
according to \cite[Propositions 2.7.3 and 2.8.2]{KL}, $A$ is isomorphic to $D_8\times f$. 
Therefore in this case 
too the subgroups of $A$ that are integrable within $A$ are 
precisely all the subgroups of $A'$, namely, $1$ and $A'\simeq 2$.\\

Case (4). $S=E_6(q)$.\\
When $A$ is nonabelian and not metacyclic then 
$A\simeq 3.(f\times 2)$ with $f>1$ even. 
Then $\abs{A'}=3$ and, trivially, the subgroups of $A$ 
that are
integrable within $A$ are precisely all the subgroups 
of $A'$.\\

This completes the proof of Theorem \ref{thm:almost_simple}.

\section{Integrable 2-homogeneous groups}\label{S-2-homog}

The enumeration of $k$-transitive and $k$-homogeneous 
subgroups of the finite symmetric group $S_n$ 
(for $k\geq 2$) has a long history, which is 
intertwined with the discovery of various finite 
simple groups, extending back to the work of Mathieu 
\cite{Mathieu1860, Mathieu1861} (or even earlier, to 
notions of transitivity investigated by Cauchy (e.g., 
\cite{Cauchy1846})). The $2$-homogeneous subgroups 
of $S_n$ 
have been completely determined as part of 
the work on the classification of the 
finite simple groups. The list of $2$-homogeneous groups in the following theorem is extracted from Blackburn and Huppert \cite[XII, Remark 7.5]{BH3} and Dixon and Mortimer
\cite[Section 7.7 and Theorem 9.4B]{DM} (with minor errors corrected). We mention several contributions to this classification. Hering's results \cite{Her74, Her85} provide the tools for the classification of $2$-transitive groups for which the socle is regular and abelian (Cases (2)-(9) below), building on significant contributions by Huppert \cite{Hup57} (who completed the solvable point stabilizer subgroup case (Case (5) below)), Livingston and Wagner \cite{LW}, and Kantor \cite{Kan69, Kan72}. Curtis, Kantor and Seitz \cite{CKS1976} classified the $2$-transitive groups with socle a nonabelian simple group of Lie type (Cases (10)-(15)). Case (16), when the socle is sporadic, was undertaken by Hering in \cite{Her85}.

We summarize the complete list of the 
$2$-homogeneous subgroups of 
$S_n$ as follows.
\begin{theorem}\label{t:2-homo_classification}
Let $G$ be a $2$-homogeneous subgroup of $S_n$ 
and set $S=\soc(G)$ to be its socle. Then $G$ is 
one of the groups appearing in the following 
list:
\begin{itemize}
\item $G$ is not $2$-transitive.\\ 
This happens precisely when: 
\begin{enumerate}
\item[(1)]$G$ is a $2$-homogeneous subgroup 
of the affine semilinear group that contains the special affine linear group, 
$\ASL_1(q)\leq G \leq \AGammaL_1(q)$ and 
$n=q\equiv 3 \pmod{4}$. 
\end{enumerate}
\item $G$ is $2$-transitive and $S$ is regular and 
abelian.\\ 
Then $G\leq \AGammaL_d(q)$, with degree 
$n=q^d=\abs{S}$ and point stabilizer $G_0$ that acts 
transitively on the set of nonzero vectors in the 
underlying vector space. This case happens precisely 
when $G$ is the semidirect product 
$G=S\semid G_0$ and one of the following holds:
\begin{enumerate}
\item[(2)] $\SL_d(q)\leq G_0\leq \GammaL_d(q)$.
\item[(3)] $G_0 \leq \GammaL_d(q)$ and $G_0$  
contains a copy of $\Sp_d(q)$ as a normal subgroup. 
\item[(4)] $G_0\leq \GammaL_6 (2^f)$ and $G_0$ contains
as a normal subgroup respectively a copy of $G_2(2^f)$  
when $f>1$, and of $\PSU_3(3)$ when $f=1$. 
\item[(5)] $G_0$ is a solvable subgroup of $\GammaL_d(q)$ 
containing a 
normal extraspecial subgroup $E$ of order $2^{df+1}$ 
such that $C_{G_0}(E) = Z(G_0)$ 
and $G_0/EZ(G_0)$ is faithfully represented on $E/Z(E)$. 
Moreover,
this situation happens if and only if 
$G_0\leq \GL_d(q)$ where:
$$(d,q)\in \set{(2,3), (2,5), 
(2,7), (2,11), (2,23), (4,3)}.$$ 
\item[(6)]
$(G_0)''\simeq \SL_2(5)$ for $d =2$ and 
$q\in\set{9, 11, 19, 29, 59}$.
\item[(7)] $G_0\simeq A_6$, $d= 4$,  $q = 2$.
\item[(8)] $G_0\simeq A_7$, $d = 4$,  $q = 2$.
\item[(9)] $G_0\simeq \SL_2(13)$, $d = 6$,  
$q = 3$.
\end{enumerate}
\item $G$ is $2$-transitive and $S$ is a nonabelian 
simple group.\\ This case happens precisely when:
\begin{enumerate}
\item[(10)] $S=A_n$ and $G\in\set{A_n,S_n}$. 
\item[(11)]  
$S=\PSL_d(q)\leq G\leq \PGammaL_d(q)$ of degree 
$n=(q^d-1)/(q-1)$.
\item[(12)]
 $S=G=\PSp_{2m}(2)$ of two possible degrees 
 $n\in\set{2^{m-1}(2^m\pm 1)}$.
\item[(13)]  $S=\PSU_3(q)\leq G\leq \PGammaU_3(q)$ 
of degree $n=q^3+1$.
\item[(14)]  $S=\null^2B_2(q)\leq G\leq \Aut(S)$,  with 
$q=2^{2m+1}$ and degree $n=q^2+1$. 
\item[(15)] $S=\null^2G_2(q)\leq G\leq \Aut(S)$,  with 
$q=3^{2m+1}$ and degree $n=q^3+1$. 
\item[(16)] $S$ is isomorphic to one of:
\begin{enumerate}
\item  $M_{11}, M_{12}, M_{23}, M_{24}$ of degree, respectively,  
$11, 12, 23$ and $24$. Moreover $G=S$. 
\item $\PSL_2(11)$ of degree $11$, and $G=S$. 
\item $A_7$ or $A_8\simeq \PSL_4(2)$, both of degree $15$, and $G=S$. 
\item $HS$ of degree $176$ and $G\in\set{S,\Aut(S)=S.2}$. 
\item $Co_3$ of degree $276$ and $G=S$.
\end{enumerate}
\end{enumerate}
\end{itemize}
\end{theorem}

\begin{rem}
Note that whenever in the statement of 
Theorem \ref{t:2-homo_classification} we 
write $S \le G\le H$, for some suitable 
$H \le S_n$, then the subgroup $H$ 
coincides with $N_{S_n}(S)$.
This follows from the fact that $S \unlhd H$ for all the groups listed above,
and so $H \le N_{S_n}(S)$. Moreover,
any group containing a $2$-homogeneous 
subgroup must itself be $2$-homogeneous, hence  $N_{S_n}(S)$ is a $2$-homogeneous subgroup of $S_n$ with socle $S$ and it must
be contained in the same $H$ coming from the classification in Theorem
\ref{t:2-homo_classification}, that is, $N_{S_n}(S) \le H$.
A similar situation occurs in Case (2) with $G_0$ in place of $G$.
\end{rem}

Note that trivially if $G$ is $2$-homogeneous 
and $G=H'$ 
(for some $H\leq  S_n$), then $H$ itself is 
also $2$-homogeneous. 
Thus our approach to solving Problem 
\ref{p:2-homo} is to determine which 
derived subgroups of $2$-homogeneous 
subgroups of $S_n$ are themselves 
$2$-homogeneous. To obtain such a 
classification we use Lemma \ref{l:basic} (both 
parts (1) and (2)). 
The following result guarantees that 
the hypotheses of Lemma \ref{l:basic}
are satisfied. Indeed, it is a significant step 
in the classification of the $2$-homogeneous subgroups of $S_n$ (that is, of the 
proof of Theorem \ref{t:2-homo_classification}). 
\begin{lemma}\label{l:2-homo_unique_min_nor}
Let $G$ be a $2$-homogeneous subgroup of 
$S_n$. Then the socle 
of $G$ is the unique minimal normal subgroup of $G$. 
\end{lemma}
\begin{proof}
This follows immediately by 
\cite[Theorem 4.1B]{DM} 
when $G$ is $2$-transitive and $S$ is 
nonabelian, and in the other cases from 
\cite[Theorem 4.3B]{DM} 
and the fact that $G$ is primitive 
(\cite[Exercise 2.1.10]{DM}; see also 
\cite[p. 402]{LW}).  
\end{proof}
As an immediate application of Lemma 
\ref{l:basic} we 
therefore have the following corollary.
\begin{cor}\label{c:2-homo}
Let $G$ be a $2$-homogeneous subgroup 
of $S_n$ having socle $S$. 
Then $G$ is integrable within $S_n$
if and only if $G/S$ is integrable 
within $N_{S_n}(S)/S$.
Suppose further that $G$ has a nontrivial 
characteristic perfect subgroup $K$. 
Then $G$ is integrable within $S_n$ if 
and only if $G/K$ is integrable within 
$N_{S_n}(K)/K$.
\end{cor}

To classify the $2$-homogeneous subgroups of degree
$n$ that are integrable within $S_n$ (and therefore  
prove Theorem \ref{thm:main_int_2-homo})  
we now take into consideration 
all the groups $G$ in the list of  
Theorem \ref{t:2-homo_classification} and 
apply Corollary \ref{c:2-homo} to them.
In the following we will always denote with 
$q=p^f$ ($p$ a prime) the order of the field 
on which classical and Lie type groups are 
defined, except when the groups are unitary 
in which case we set $q^2=p^f$.\\

\noindent 
Case (1).
We have that $\AGammaL_1(q)/\ASL_1(q)$ is 
isomorphic to a metacyclic group of order 
$(q-1):f$. By Lemma \ref{l:metacyclic} we 
conclude that in this case $G$ is integrable 
within $S_n$ precisely when $G$ 
is a subgroup of $(\AGammaL_1(q))'$, 
containing $\ASL_1(q)$.\\

\noindent 
Case (2). We have that 
$\ASL_d(q)=S\semid \SL_d(q)$ 
is a characteristic and perfect subgroup of 
$G=S\semid G_0$, and therefore by 
Corollary \ref{c:2-homo} the 
integrability of $G$ within $S_n$ is 
equivalent to the integrability 
of $G_0/\SL_d(q)$ within  
$\GammaL_d(q)/\SL_d(q)$. 
Note that this latter group is 
metacyclic isomorphic to
$(q-1):f$. 
Therefore Lemma \ref{l:metacyclic} implies 
that $G$ is integrable within $S_n$ if and 
only if 
$\SL_d(q)\leq G_0\leq \GammaL_d(q)'.$ \\

\noindent 
Case (3). Let $H$ be a subgroup of 
$\GammaL_d(q)$ isomorphic to $\Sp_d(q)$ and 
let $H\leq G_0\leq 
N_{\GammaL_d(q)}(H)\simeq 
\mathrm{\Gamma}\Sp_d(q)$.
By \cite[Section 2.4]{KL} this latter group 
is isomorphic to 
$$\big(\Sp_d(q)\!:\!(q-1)\big)\!:\!f.$$
In particular 
$\Sp_d(q)$ is characteristic in 
$\mathrm{\Gamma}\Sp_d(q)$, and thus by 
Corollary \ref{c:2-homo} the integrable 
$2$-homogeneous subgroups of 
this case correspond precisely to those $G_0$ 
such that $G_0/H$ is isomorphic to an 
integrable subgroup within 
$\mathrm{\Gamma}\Sp_d(q))/\Sp_d(q)$. 
Finally, note that 
$\mathrm{\Gamma}\Sp_d(q))/\Sp_d(q)$  is a 
metacyclic group. Lemma \ref{l:metacyclic} 
completes the proof of this case.\\

\noindent 
Case (4). In this case $q=2^f$. 
It is well-known that 
the group $\Sp_6(q)$ contains, as a maximal 
subgroup, a subgroup isomorphic to $G_2(q)$ 
(see \cite{Dickson1901} or  
\cite[Theorem 3.7]{Wil}), 
which is transitive on the corresponding 
projective space.
We let $H$ be a subgroup of $\SL_6(q)$
isomorphic to $G_2(q)$ and we first determine 
the structure of $N_{\GammaL_6(q)}(H)$. \\
By \cite[Table 8.29]{BHR}, 
we see that a field automorphism $\phi$ of 
order $f$ can be taken to normalize $H$. 
Thus by Dedekind's modular law we have that
$$N_{\GammaL_6(q)}(H)=N_{\GL_6(q)}(H).\seq{\phi}.$$
Also $N_{\GL_6(q)}(H)=ZN_{\SL_6(q)}(H)$,  
where $Z=Z(\GL_6(q))\simeq (q-1)$ 
(here we used again \cite[Table 8.29]{BHR}).
We now claim that 
$$N_{\SL_6(q)}(H)=
Z(\SL_6(q))\times H\simeq (q-1,3)\times 
G_2(q).$$ 
This follows from the fact that every maximal 
subgroup of $\SL_6(q)$ containing a copy of 
$G_2(q)$ is isomorphic to 
$(q-1,3)\times \Sp_6(q)$. 
To prove this statement we look at the 
maximal subgroups of 
$\SL_6(q)$ in 
\cite[Tables 8.24 and 8.25]{BHR} (eventually
also Tables 8.3 and 8.4 to exclude the unique 
non immediate case of maximal subgroups 
isomorphic to $\SL_3(q^2).(q+1).2$). 
This shows that 
$$N_{\GammaL_6(q)}(H)\simeq 
(H\times (q-1)).f$$
and therefore $N_{\GammaL_6(q)}(H)/H$ is 
metacyclic.
Lemmas \ref{l:basic} and \ref{l:metacyclic} 
complete the proof of this case. We observe that Remark \ref{rem:thmB+exceptions} summarizes the integrable subgroups for this particular case.\\

\noindent 
Case (5). 
The main reference for this case is 
\cite{Hup57}, where the full classification 
for the $2$-transitive groups of this type 
appear. In particular we have that $E$ is a 
subgroup of $\SL_d(q)$; moreover, since $q$ 
is always a  prime number, 
$G_0\leq \GL_d(q)=\GammaL_d(q)$. 
Now, the integrability of $G$ in $S_n$, by 
Corollary \ref{c:2-homo}, is equivalent to 
the integrability of $G_0$ within $\GL_d(q)$. 
Thus in particular $G_0$ must be a subgroup 
of $\SL_d(q)$ that normalizes $E$. 
We examine all possible situations that are 
classified in \cite{Hup57}.\\
When $d=2$, according to \cite{Hup57}, the subgroups of $\SL_2(q)$ 
that are transitive on nonzero vectors 
appear just when $q\in\set{3,5}$.  
Moreover, if $q=3$, since  
$\SL_2(3)\simeq 2\!\udot\! A_4$, the 
subgroup $E$ is the unique Sylow $2$-subgroup 
of $\SL_2(3)$. Then either $G_0=E$ or 
$G_0=\SL_2(3)$. Both are integrable in 
$\GL_2(3)$ since $\SL_2(3)=(\GL_2(3))'$ and 
$E=(\GL_2(3))''$. \\
When $q= 5$ then $E$ is a Sylow $2$-subgroup 
of $\SL_2(5)$, since $\SL_2(5)\simeq 
2\!\udot\! A_5$. 
In this case $G_0=N_{\SL_2(5)}(E)
\simeq 2\!\udot\! A_4\simeq \SL_2(3)$ is 
transitive on nonzero vectors and also 
integrable (in $\GL_2(5)$) since it is equal to 
$(N_{\GL_2(5)}(E))'$.\\
Finally we consider the case $d=4$ and $q=3$. 
Then 
%
%
$E$ is isomorphic to the central product 
$D_8\circ Q_8$. Also $Z(E)=Z(\GL_4(3))$ and 
up to conjugation there are only three 
subgroups in $\GL_4(3)$ that are transitive 
on nonzero vectors. They are all contained in 
$\SL_4(3)$ and they are respectively 
isomorphic to $E:5$, $(E:5).2$ and 
$(E:5).4$. Only the first of these is 
integrable being the derived subgroup of 
the other two.\\

\noindent 
Case (6). Set $L=(G_0)''$ be the second 
derived subgroup of $G_0$, 
so that $L\simeq \SL_2(5)\simeq 
2\!\udot\! A_5$ and 
$L\leq G_0\leq N_{\GammaL_2(q)}(L)$ 
with $q$ as in the statement. 
By Corollary \ref{c:2-homo}\red{,}  $G$ is 
integrable within $S_n$ if and only if 
$G_0/L$ is integrable within 
$N_{\GammaL_2(q)}(L)/L$, therefore we 
examine the structure of this last group. 
By \cite[Table 8.2]{BHR}, $L$ is always a 
maximal subgroup $\SL_2(q)$ and therefore it 
coincides with its normalizer in $\SL_2(q)$. 
When $q\neq 9$ we have that $q$ 
is prime, forcing $\GL_2(q)=\GammaL_2(q)$. 
It follows that $N_{\GammaL_2(q)}(L)=LZ$, 
where $Z=\seq{z}$ is the center of $\GL_2(q)$ 
(where we used again 
\cite[Table 8.2]{BHR}).
The 
conclusion is that, when $q\neq 9$, only 
$G=S\semid L\simeq S:\SL_2(5)$ is integrable 
within $S_n$. Assume now that $q=9$. Again 
by \cite[8.2]{BHR}, the subgroup $L$ is 
maximal in $\SL_2(9)$ and, outside 
$\SL_2(9)Z$, 
only field automorphisms of order two 
normalize it. Therefore 
$N_{\GammaL_2(9)}(L)=LZ.2$ and 
$N_{\GammaL_2(9)}(L)/L\simeq D_8$. 
We conclude that, for $q=9$,
$G_0$ is integrable within $S_n$ if it is
either $L$ or 
$L\seq{z^4}=(N_{\GammaL_2(9)}(L))'$.\\

\noindent 
Cases (7)-(9). These groups are all 
integrable, since $G_0$ 
(and therefore $G$) is perfect.\\

\noindent 
Case (10). Of course $G=A_n$ is integrable within $S_n$,  while 
$G=S_n$ is not.\\

\noindent 
Case (11). Since $S=\PSL_d(q)$ and 
$N_{S_n}(S)=\PGammaL_d(q)$, 
by \cite[Proposition 2.2.3]{KL}, the group 
$\PGammaL_d(q)/\PSL_d(q)\leq \Out(S)$
is isomorphic to a finite metacyclic group 
$m:f$, where $m=\gcd(d,q-1)$. 
Lemma \ref{l:metacyclic} together with Corollary \ref{c:2-homo}
imply that 
the integrable subgroups in this case are 
precisely all those $G$ such that 
$\PSL_d(q)\leq G \leq (\PGammaL_d(q))'$.\\


\noindent 
Case (12). Here we have that $G=S=N_{S_n}(S)$ 
is a perfect group and therefore trivially 
integrable.\\

\noindent 
Case (13). Now we write $q^2=p^f$, so that 
$S = \PSU_3(q)$ and 
$N_{S_n}(S) = \PGammaU_3(q)$.
By \cite[Proposition 2.3.5]{KL},
the group $\PGammaU_3(q)/\PSU_3(q)$ is 
isomorphic to a finite metacyclic group 
$m:(2f)$, where $m=\gcd(3,q+1)$. 
By Corollary \ref{c:2-homo},  
the integrable subgroups in this case are 
precisely  $S=\PSU_3(q)$ (always) and 
$\PGU_3(q)$ when $3\vert (q+1)$ 
and 
$\PGammaU_3(q)/\PSU_3(q)$ is not abelian. 
This latter situation happens if and only 
if $p\equiv 2 \pmod 3$ and $f/2$ is odd. \\


\noindent 
Cases (14)-(15). In both cases we have that 
$\Out(S)$ is a cyclic group of order $f$. 
Therefore, by Lemma \ref{l:metacyclic}, Theorem \ref{thm:almost_simple} and
Lemma \ref{l:almost-simple_1}, the only integrable groups are 
 $G=S$ in both cases. \\

\noindent 
Case (16).  All the simple 
groups $S$ listed have trivial or 
cyclic outer automorphism group. 
Therefore,  by Lemma \ref{l:metacyclic}, Theorem \ref{thm:almost_simple} and
Lemma \ref{l:almost-simple_1}, these are the unique 
integrable subgroups in this case.\\


It is straightforward now to check that in 
all cases - except Case (5) - the $2$-homogeneous subgroups $G$ 
that are integrable within $S_n$ are precisely 
those $G$ such that 
$$(N_{S_n}(S))''\leq 
G\leq (N_{S_n}(S))',$$
thus proving Theorem \ref{thm:main_int_2-homo}. 
We leave this verification to the reader. We collect in the following remark 
the solvable exceptions appearing in Case (5).
\begin{rem}\label{rem:thmB+exceptions}
Let $G$ be a $2$-transitive subgroup of $S_n$ that 
belongs to Case (5) of Theorem 
\ref{t:2-homo_classification}. 
Note that this is the 
unique case in which $G$ is solvable.
Then $G$ is integrable within 
$S_n$ precisely when $G$ is isomorphic 
to a group in the following table

    \begin{center}
    \begin{tabular}{l|c}
        {\rm isomorphism type of $G$} & {\rm degree $n$} \\
        \hline \\
    $3^2:Q_8$     & $3^2$\\
    $\ASL_2(3)$   & $3^2$\\
    $5^2:\SL_2(3)$ & $5^2$\\
    $3^4:((D_8\circ Q_8):5)$ & $3^4$
    \end{tabular}
\end{center}

Note that each of these subgroups is 
between $(N_{S_n}(S))''$ and 
$(N_{S_n}(S))'$ and that, apart from when 
$n=3^2$, there are also subgroups in this interval 
that are not integrable (within $S_n$). 
\end{rem}

As mentioned in the Introduction, the following 
example shows that the more general 
problem of classifying all integrable $2$-homogeneous 
groups (as abstract groups) has a different solution 
than Problem \ref{p:2-homo} has.
\begin{exa}\label{ex:PGammaL_3(7)}
{\rm The group $G=\PGL_3(7)$
acts $2$-transitively on the set of projective 
points (or, equivalently, on  the set of projective 
lines) whose size is $57$ (this group arises in 
Case (11) of Theorem \ref{t:2-homo_classification}). 
By Theorem \ref{thm:main_int_2-homo}, $G$ is 
not integrable within $S_{57}$. Nevertheless, the 
simple group $\PSL_3(7)$ has outer automorphism group 
isomorphic to $S_3$, since it is generated by an 
outer diagonal automorphism of order three and a graph 
involution. This implies that 
$(\Aut(\PSL_3(7)))'=\PGL_3(7)$ which demonstrates 
the abstract integrability of $G$.}
\end{exa}

\section*{Acknowledgements}
We authors would like to thank Peter Cameron for helpful comments on an earlier draft of the paper.

\end{document}